\documentclass[]{interact}
\usepackage{geometry}
\usepackage{epstopdf}
\usepackage{enumitem}   
\usepackage[caption=false]{subfig}
\usepackage{bm}
\usepackage{MnSymbol}
\usepackage{soul}
\usepackage{setspace}
\usepackage{tikz-cd}
\usepackage{tikz}
\usepackage{hyperref}
\usepackage[backend=biber,backref=true]{biblatex}
\usepackage{comment}

\theoremstyle{plain}
\newtheorem{theorem}{Theorem}[section]
\newtheorem{lemma}[theorem]{Lemma}
\newtheorem{corollary}[theorem]{Corollary}
\newtheorem{proposition}[theorem]{Proposition}

\theoremstyle{definition}
\newtheorem{definition}[theorem]{Definition}

\theoremstyle{remark}
\newtheorem{remark}{Remark}

\begin{document}


\title{\vspace{-0.5 cm}Mass and rigidity in almost K\"ahler geometry}
\author{
\name{Partha Ghosh}
}
\newcommand{\Addresses}{{
  \bigskip
  \footnotesize

  \textsc{Sorbonne Université, Université Paris Cité, CNRS, IMJ-PRG, F-75005 Paris, France}\par\nopagebreak
  \textit{E-mail address}: \texttt{Partha.Ghosh@imj-prg.fr}
}}
\maketitle
\begin{abstract}
    We derive an explicit formula for the ADM mass of asymptotically locally Euclidean (ALE) almost Kähler manifolds. The formula expresses the mass in terms of the total Hermitian scalar curvature and topological data associated with the underlying almost complex structure, extending a result of Hein and LeBrun \cite{HeinLebrun} in the K\"ahler ALE case. Our approach is based on a spin$^\mathbb {C}$ adaptation of Witten's proof of the positive mass conjecture in the spin case and is therefore distinct from previous complex-geometric methods. In dimension $4$, we prove a positive mass theorem and a Penrose-type inequality for asymptotically Euclidean (AE) almost Kähler manifolds. We also study rigidity phenomena of almost K\"ahler ALE manifolds. We prove that an almost K\"ahler-Einstein ALE manifold with nonnegative scalar curvature and certain decay assumptions is necessarily Kähler-Einstein. In particular, any four dimensional Ricci-flat almost Kähler manifold with maximal volume growth and curvature in $L^2$ is Kähler, yielding new evidence towards the Bando--Kasue--Nakajima conjecture \cite{BKN}. We also discuss analogous rigidity results for asymptotically locally flat (ALF) manifolds.
\end{abstract}
\tableofcontents
\section{Introduction}
\subsection{Background} Defining mass in general relativity is a very subtle issue. Roughly speaking, in Minkowski spacetime, we say that the energy content is given by an integral of the energy density of all the fields being considered over a constant time hypersurface. The problem is that, while this construction can still be carried out
for matter fields, it fails for the gravitational field. This is because the gravitational field is now described
by the curvature of spacetime and there is no well-defined local concept of energy density in that setting.\par
This is where the concept of ADM mass enters. Introduced by Arnowitt, Deser, and Misner \cite{ADM} in the early 1960s through a Hamiltonian formulation of general relativity, the ADM mass encapsulates all the properties of mass one would expect.\par 
A fundamental desired property of mass is its positivity. The Riemannian positive mass theorem states that any complete asymptotically flat (AF) Riemannian manifold with non-negative scalar curvature has non-negative ADM mass, with equality if and only if the manifold is isometric to Euclidean space. The first rigorous proof was established by Schoen and Yau \cite{ScYau1,ScYau2} in dimensions up to $7$, and independently by Witten \cite{Witten} in all dimensions for spin manifolds.\par
While the ADM mass is classically defined for AF manifolds, analogous notion of mass remains well defined under other curvature decay conditions, notably in the asymptotically locally Euclidean (ALE) setting. In the K\"ahler ALE case, Hein and LeBrun \cite{HeinLebrun} derived a transparent and beautiful mass formula expressing the ADM mass purely in terms of complex-geometric data, relating it to integrals of the scalar curvature and topological quantities. This article is devoted to derive a similar mass formula in the almost Kähler ALE case, which recovers Hein and LeBrun's formula in the K\"ahler case. The proof adapts Witten's proof of the positive mass conjecture for the spin case in the spin$^\mathbb{C}$ case (notice almost K\"ahler implies spin$^\mathbb{C}$) and hence different from Hein and LeBrun's approach. We also establish a positive mass theorem and a Penrose-type inequality for four-dimensional almost Kähler asymptotically Euclidean manifolds, following LeBrun’s \cite{LeBrun100} approach, which draws on results from four-dimensional symplectic topology. We also deduce that if the ALE metric is almost Kähler-Einstein with certain decay properties and nonnegative scalar curvature, then it is in fact Kähler-Einstein. Similar rigidity statements are also discussed for asymptotically locally flat (ALF) manifolds.
\subsection{Overview of the main results}
\begin{definition}\label{almost kahler}
    An almost-K\"ahler manifold is an almost-Hermitian manifold $(M,g,J)$ whose fundamental two form $\omega(u,v)=g(Ju, v)$ is closed.
\end{definition}
In the definition $g$ is the Riemannian metric, it preserves the almost complex structure $J$. In other words, an almost-K\"ahler manifold is a symplectic manifold equipped with a compatible metric.
\begin{definition}\label{ALE}
    An $n$-dimensional complete Riemannian manifold $(M,g)$ is said to be \textit{asymptotically locally Euclidean} (ALE) of order $\tau>0,$ if there exists a compact subset $K\subset M$ such that $M\setminus K$ has coordinates at infinity; namely there are $R>0,$ a finite group $\Gamma\subset O(n)$ acting freely on $\mathbb{R}^n\setminus B_R(0),$ and a $C^\infty$-diffeomorphism $\mathcal{X}:M\setminus K\rightarrow \big(\mathbb{R}^n\setminus\overline{B_R(0)}\big)/\Gamma $ such that $\phi=\mathcal{X}^{-1}  \pi$ satisfies ($\pi$ is the natural projection $\pi:\mathbb{R}^n\rightarrow \mathbb{R}^n/\Gamma$)
    \begin{align*}
        (\phi^*g)_{ij}(x)=\delta_{ij}+a_{ij}(x)\hspace{1 ex}\text{  for }x\in \mathbb{R}^n\setminus\overline{B_R(0)}
    \end{align*}
    where $|\partial ^p a_{ij}(x)|=O (|x|^{-\tau-p})$ for all $p\geq 0.$
\end{definition}
From now on we will write $r$ for $|x|$, where $x$ denote the coordinate at the \textit{infinity} $M_\infty:=M\setminus K$ as described in the definition above. We follow Hein and LeBrun's \cite{HeinLebrun} convention to define the ADM mass.
\begin{definition}\label{mass}
    Let $(M,g)$ be an $n$-dimensional ALE manifold of order $\tau>\frac{n-2}{2}$ and integrable scalar curvature. Then the ADM mass of $(M,g)$ is defined as 
    \begin{align}
        \mathfrak{m}(M,g):=\lim_{r\longrightarrow\infty}\frac{\mathbf{\Gamma}(\frac{n}{2})}{4(n-1)\pi^{\frac{n}{2}}}\int_{S_r/\Gamma}\big(\partial_i(\phi^*g)_{ij}(x)-\partial_j(\phi^*g)_{ii}(x)\big)\partial_j \righthalfcup dx
    \end{align}
\end{definition}
where $\phi$ is as in definition \ref{ALE}, $S_r$  is the Euclidean coordinate sphere of radius $r$. Bartnik \cite{Bartnik} and Chru{\'s}ciel \cite{Chruciel} independently discovered that the ADM mass is well defined, i.e., it is independent of the choice of coordinates at infinity for $\tau>\frac{n-2}{2}$.\par
For a smooth manifold $M$, one can define the compactly supported de Rham cohomology $H^k_c(M),$ as well as the usual de Rham cohomology $H^k(M)$. If $M$ is oriented, Poincar\'e duality gives us an isomorphism $H^2_c(M)\cong [H^{2m-2}(M)]^*$. On the other hand, there is a natural map $\iota: H^2_c(M)\rightarrow H^2(M)$ induced by the inclusion of compactly supported forms into all differential forms, and in the ALE setting, this map is actually
an {isomorphism}. Using this notation we now state our explicit formula for the mass.
\begin{theorem}\label{Mass formula}
    An almost K\"ahler ALE manifold of $(M,g,J)$ dimension $n=2m\geq 4$ and of order $\tau>\frac{n-2}{2}$ with $s,|\nabla^{\mathrm{LC}}J|^2\in L^1$ has the mass given by 
    \begin{align}
        \mathfrak{m}(M,g)=\frac{(m-1)!}{4(2m-1)\pi^m}\int_M\frac{s+s^*}{2}\, d\mathrm{v}_g-\frac{\big\langle \iota^{-1}c_1(M,J),[\omega]^{m-1}\big\rangle}{(2m-1)\pi^{m-1}}\label{mass formula}
    \end{align}
    where $s$ and $s^*$ are respectively the scalar curvature and the star-scalar curvature of $g$. $c_1(M,J)$ is the first Chern class of the complex structure and $\big\langle\,,\,\big\rangle$ is the duality pairing between $H^2_c(M)$ and $H^{2m-2}(M)$.
\end{theorem}
$s^*$ the star-scalar curvature is a standard notation in almost K\"ahler geometry, defined as 
\begin{align*}
    s^*=2R(\omega,\omega)
\end{align*}
where $R$ is the Riemann curvature tensor. $s$ and $s^*$ are related by the following identity.
\begin{align*}
    s^*-s=|\nabla^{\mathrm{LC}}\,\omega|^2=\frac{1}{2}|\nabla^{\mathrm{LC}}\, J|^2
\end{align*}
$\nabla^{\mathrm{LC}}$ stands for the Levi-Civita conection. Notice, $s=s^*$, if and only if the almost K\"ahler structure is K\"ahler. The relation above shows that in the K\"ahler case, we recover the mass formula for the K\"ahler case obtained by Hein and LeBrun \cite{HeinLebrun}.
\begin{theorem}[Hein--LeBrun \cite{HeinLebrun}, LeBrun \cite{LeBrun100}] 
    An K\"ahler ALE manifold $(M,g,J)$ of dimension $n=2m\geq 4$ and of order $\tau>\frac{n-2}{2}$ with $s\in L^1$ has the mass given by 
    \begin{align*}
        \mathfrak{m}(M,g)=\frac{(m-1)!}{4(2m-1)\pi^m}\int_M s\, d\mathrm{v}_g-\frac{\big\langle \iota^{-1}c_1(M,J),[\omega]^{m-1}\big\rangle}{(2m-1)\pi^{m-1}}
    \end{align*}
\end{theorem}
The quantity $s_H:=\frac{1}{2}(s+s^*)$ is referred to in the almost K\"ahler literature as the \textit{Hermitian scalar curvature}, and we adopt the same terminology here.\par
The obvious observation is to compare theorem \ref{Mass formula} with the more familiar compact case. If $(M,g,J)$ is a \text{compact} almost K\"ahler manifold of real dimenson $n=2m$, then its total Hermitian scalar curvature is a \textit{symplectic invariant} and known to satisfy the following formula.
\begin{align}\label{Blair formula}
    \int_M \frac{s+s^*}{2} d\mathrm{v}_g=\frac{4\pi}{(m-1)!}\big\langle c_1(M,J),[\omega]^{m-1}\big\rangle
\end{align}
This was first discovered by Blair \cite{Blair}. Theorem \ref{Mass formula} tells us that the mass measures the extent to which Blair's formula \eqref{Blair formula} fails in the ALE case:
\begin{align*}
    \frac{4\pi^m(2m-1)}{(m-1)!}\mathfrak{m}(M,g)=\int_M \frac{s+s^*}{2} d\mathrm{v}_g-\frac{4\pi}{(m-1)!}\big\langle \iota^{-1} c_1(M,J),[\omega]^{m-1}\big\rangle
\end{align*}
An immediate consequence of theorem \ref{Mass formula} is the following.
\begin{corollary}
    Let $(M,g,J)$ be an almost K\"ahler ALE manifold of dimension $n\geq 4$ and of order $\tau>n-2$. If $(M,g,J)$ scalar-flat with $c_1=0$, then $(M,g,J)$ is K\"ahler.
\end{corollary}
When \( m \geq 3 \) and \( J \) is integrable, i.e., when \( (M,g,J) \) is K\"ahler, LeBrun and Hein \cite{HeinLebrun} fully exploited the tools from K\"ahler geometry to show that even if \( M \) is \emph{a priori} assumed to have many ``ends'', it in fact has only one end (note that in our definition~\ref{ALE} we assumed the manifold has a single end). Moreover, when \( (M,g,J) \) is K\"ahler AE, they expressed the topological term
\[
\big\langle \iota^{-1} c_1(M,J),[\omega]^{m-1}\big\rangle
\]
appearing in the mass as a sum of the volumes of complex hypersurfaces in \( M \), each with a positive integer multiplicity. This allowed them to establish a positive mass theorem as well as a Penrose-type inequality. The most immediate extension of the positive mass conjecture to ALE manifolds was shown to be false by LeBrun \cite{Counter}. Under some assumptions on the manifold with the ALE metric, the positive mass conjecture has been investigated in \cite{Nakajima1},\cite{Dahl}.\par
Unfortunately, the almost K\"ahler condition does not grant us access to the powerful machinery of complex and K\"ahler geometry. Instead, one must rely on tools from symplectic geometry, which are substantially more effective only in dimension \(4\). LeBrun has already pointed the way in this direction in \cite{LeBrun100}. While descending from the K\"ahler to the almost K\"ahler setting, one must be careful. A major advantage of the K\"ahler condition is that the almost-complex structure \(J\) is parallel, namely \( \nabla^{\mathrm{LC}} J = 0 \). In contrast, in the almost K\"ahler case we only have
\[
|\nabla^{\mathrm{LC}} J|^2 = 2(s^* - s) = O\bigl(r^{-(\tau+2)}\bigr)
\quad \text{as } r \longrightarrow \infty 
\]
\par 
In dimension $4$, to define the mass we need $\tau>1$ and $|\nabla^{\mathrm{LC}}J|^2\in L^1$. However the decay condition $|\nabla^{\mathrm{LC}}J|^2=O(r^{-(\tau+2)})$ for $\tau>1$ does not guarantee $|\nabla^{\mathrm{LC}}J|^2\in L^1$. Instead we will work with $|\nabla^{\mathrm{LC}}J|^2=O(r^{-\eta})$ with $\eta>4$ as $r\longrightarrow\infty$, since this condition implies $|\nabla^{\mathrm{LC}}J|^2\in L^1$ and is convenient for using the tools from symplectic geometry we need. One should view this as analogous to the assumption imposed on the scalar curvature $s$. The ALE condition only gives us $s=O(r^{-{(2+\tau)}})$ which is not enough to ensure $s\in L^1$ for $\tau>1$. But $s$ needs to be in $L^1$ for the ADM mass to be well defined and hence usually an extra assumption of $s$ being in $L^1$ is made. Since we are working with almost K\"ahler manifolds, similar assumption on the Hermitian scalar curvature $s_H$ is quite natural, $s_H\in L^1$ is also required if we want to make sense of the right hand side of the equation \eqref{mass formula}. $s,s_H\in L^1$ implies $|\nabla^{\mathrm{LC}}J|^2\in L^1$. Therefore the stronger decay condition of $|\nabla^{\mathrm{LC}}J|^2$ is quite natural and satisfied by the manifolds constructed in theorem \ref{lemma existence}. \par
In the definition of an ALE manifold \ref{ALE}, we assumed that the manifold only has one end. In general, one can talk about ALE manifolds with multiple ends, i.e., the complement of a compact set $M\setminus K$ can have finitely many components, each of which is diffeomorphic to a quotient $\Big(\mathbb{R}^n\setminus\overline{B_R(0)}\Big)/\Gamma_i$ and at each end the metric $g$ again becomes the Euclidean metric plus error terms that fall off sufficiently rapidly at infinity. It is assumed that the reader is able to comprehend the rigorous definition of ALE with multiple ends from the definition with one end \ref{ALE}. For the next theorem, we consider an ALE metric that may have multiple ends.
\begin{theorem}\label{16}
    Let $(M,g,J)$ be an almost K\"ahler ALE manifold of dimension $4$ and of order $\tau>1$ and we allow that $M$ may have multiple ``ends". As $r\longrightarrow\infty$, if $|\nabla^{\mathrm{LC}}J|^2=O(r^{-\eta})$ with $\eta>4$, then $M$ has exactly one end.
\end{theorem}
\begin{theorem}\label{17}
    Let $(M,g,J)$ be an almost K\"ahler ALE manifold of dimension $4$ and of order $\tau>1$. If as $r\longrightarrow\infty$, $|\nabla^{\mathrm{LC}}J|^2=O(r^{-\eta})$ with $\eta>4$, then $M$ is diffeomorphic to the complement of a tree of symplectically embedded $2$-spheres in a rational complex surface.
\end{theorem}
\begin{theorem}\label{18}
    Let $(M,g,J)$ be an almost K\"ahler ALE manifold of dimension $4$ and of order $\tau>1$. If as $r\longrightarrow\infty$, $|\nabla^{\mathrm{LC}}J|^2=O(r^{-\eta})$ with $\eta>4$, then the fundamental group of $M$ is finite.
\end{theorem}
\begin{theorem}[Penrose inequality]\label{Penrose} Let $(M,g,J)$ be an almost K\"ahler AE manifold of dimension $4$ and of order $\tau>1$. If $s\geq 0$ and as $r\longrightarrow\infty$, $|\nabla^{\mathrm{LC}}J|^2=O(r^{-\eta})$ with $\eta>4$, then $(M,J)$ carries a pseudoholomorphic curve $D$ that is expressed as a sum $\sum n_i D_i$ of compact hypersurfaces of real codimension $2$ with positive integer coefficients, with the property that $\bigcup_i D_i\neq \emptyset$ whenever $(M,J)$ is not diffeomorphic to $\mathbb{R}^4$. In terms of the curve, the mass of the manifold then satisfies 
\begin{align*}
    \mathfrak{m}(M,g)\geq \frac{1}{3\pi}\sum_i n_i \,\mathrm{vol}(D_i)
\end{align*}
and equality holds if and only if $(M,g,J)$ is scalar-flat K\"ahler.
\end{theorem}
This immediately gives us the following.
\begin{theorem}[positive mass theorem]\label{PMT}
    Let $(M,g,J)$ be an almost K\"ahler AE manifold of dimension $4$ and of order $\tau>1$. If $s\geq 0$ and as $r\longrightarrow\infty$, $|\nabla^{\mathrm{LC}}J|^2=O(r^{-\eta})$ with $\eta>4$, then $\mathfrak{m}(M,g)\geq 0$, with equality iff $(M,g)$ is Euclidean space.
\end{theorem}
Almost K\"ahler is another way of saying that the Riemannian manifold admits a symplectic form which is compatible with metric. Therefore, there are plenty of examples of almost K\"ahler ALE manifolds which are not K\"ahler. One way of constructing them is to make compactly supported deformations to K\"ahler ALE manifolds which exist plenty in the literature, e.g., see \cite{Kronheimer1}, \cite{Counter}, \cite{Lock}. We give a brief proof of this construction in the next section.
\begin{theorem}\label{lemma existence}
Let $(M,g_0,J_0)$ be a K\"ahler manifold. Let's say $\omega_0$ is the K\"ahler form. Then there exists a smooth compactly supported deformation $J_t$ of $J_0$ such that for all sufficiently small $t\neq0$:
\begin{enumerate}
\item $J_t$ is $\omega_0$-compatible,
\item $J_t$ is not integrable.
\end{enumerate}
Consequently $(M,g_t,J_t)$ is strictly almost K\"ahler, where $g_t(u,v):=\omega_0(u,J_t v)$.
\end{theorem}
Under the extra assumption of the almost K\"ahler metric being Einstein, this is in the realm of the so-called Goldberg conjecture \cite{Goldberg}: \textit{every compact almost-K\"ahler Einstein manifold is actually K\"ahler}.  It should be emphasized that, although Goldberg did not clearly stipulate that the manifolds in question are required to be compact, this is taken to be the case. In fact one can construct counter-examples to the conjecture in the non-compact case in all dimensions $\geq 4$ using Tod's ansatz \cite{Armstrong}. In the compact case, Sekigawa \cite{Seikigawa1, Sekigawa2} has shown that the conjecture is actually true for metrics with $s\geq 0$.\par 
We show that Sekigawa's proof \cite{Sekigawa2} actually goes through in the non-compact case under suitable decay properties of the metric and $\nabla^{\mathrm{LC}}J$ at ``\textit{infinity}" yielding the following. 
\begin{theorem}\label{dim >4}
    Let $(M,g,J)$ be an almost-K\"ahler Einstein ALE manifold of dimension $n\geq 4$ and of order $\tau>\frac{2}{3}(n-4)$ with $s\geq 0$ and as $r\longrightarrow\infty$, $|(\nabla^{\mathrm{LC}})^2J|=O (r^{-\eta})$ with $\eta\geq \frac{2}{3}(n-1)$. Then $(M,g,J)$ is K\"ahler Einstein.
\end{theorem}
In dimension $4$, LeBrun \cite{LeBrun} has recently established a stronger statement in the compact case. His result says that \textit{any compact almost K\"ahler $4$-manifold $(M,g,J)$ with $s\geq 0$ and $\delta W_+=0$ is actually a constant-scalar-curvature K\"ahler manifold}, where $\delta$ denotes the divergence operator and $W_+$ is the self-dual Weyl curvature.  When $g$ is an Einstein metric, this recovers Sekigawa's \cite{Seikigawa1} partial solution to the $4$-dimensional Goldberg conjecture. In the noncompact setting, LeBrun's proof carries over under appropriate decay conditions, leading to the following theorem.
\begin{theorem}\label{theorem 2}
    Let $(M,g,J)$ be a $4$-dimensional almost-K\"ahler ALE manifold of order $\tau>0$. If $s\geq 0$ and $\delta W_+=0$, then $(M,g,J)$ is a constant-scalar-curvature K\"ahler (cscK) manifold.
\end{theorem}
Here $W_+$ denotes the self-dual Weyl curvature of the metric $g$. In particular when $g$ is Einstein we get
\begin{theorem}\label{dim = 4}
    Let $(M,g,J)$ be a $4$-dimensional almost-K\"ahler Einstein ALE manifold of order $\tau>0$ with $s\geq 0$. Then $(M,g,J)$ is K\"ahler Einstein.
\end{theorem}
This discussion is particularly relevant in the Ricci-flat case, namely for \textit{gravitational instantons}. A classification of $4$-dimensional hyper-K\"ahler ALE $4$-manifolds was given by Kronheimer \cite{Kronheimer1,Kronheimer2}. A classification of K\"ahler ALE spaces \cite{Suviana, Wright} are also known to be the hyper-K\"ahler ones and finite free quotients of them. A conjecture by Bando--Kasue--Nakajima \cite{BKN} says that \textit{all $4$-dimensional Ricci-flat manifolds with maximal volume growth and curvature in $L^2$ are K\"ahler}. In other words, \textit{all simply connected $4$-dimensional Ricci-flat manifolds with maximal volume growth and curvature in $L^2$ are hyper-K\"ahler}. Maximal volume growth and curvature in $L^2$ in dimension $4$ implies the metric is ALE of order $4$ \cite{BKN}.\par 
Nakajima \cite{Nakajima1} showed that the conjecture is true for the spin case if $\Gamma\subseteq SU(2)$. Theorem \ref{dim = 4} provides an improvement on the current state of the conjecture in the following sense.
\begin{theorem}\label{important}
    A $4$-dimensional Ricci-flat almost K\"ahler manifold with maximal volume growth and curvature in $L^2$ is K\"ahler.
\end{theorem}
\begin{remark}\label{r1}
The Einstein assumption gives us Ric$=\frac{s}{n}g$. Hence by Bonnet--Myers an ALE Einstein metric can not have positive scalar curvature. So, only the Ricci-flat case in theorems \ref{dim >4} and \ref{dim = 4} are not vacuous. Moreover with the extra assumption: $s\in L^1$, it also cancels out the negative Einstein constant case.\par 
It is worth clarifying the extent to which theorem \ref{important} provides new information. For $\Gamma\subset SU(2)$, the symplectic fillings of $S^3/\Gamma$ with respect to the standard contact structure are diffeomorphic to the blow-ups of the minimal resolution of $\mathbb{C}^2/\Gamma$ (\cite{McDuff, OO-1, OO-2}). But the minimal resolution saturates the ALE Hitchin--Thorpe inequality \cite{Nakajima1} and each blow up violates the inequality. Consequently no non-trivial blow-up can admit a Ricci-flat ALE metric. Hence, within the class relevant to theorem \ref{important}, the case $\Gamma\subset SU(2)$ appears to be already constrained by the known classification of symplectic fillings together with the ALE Hitchin--Thorpe inequality. In particular, a genuinely new non-vacuous application of theorem \ref{important}, if one exists, should be sought in the case $\Gamma\subset U(2)\setminus SU(2)$. We do not know at present whether such an example exists.
\end{remark}
The analogous rigidity statements also hold for asymptotically locally flat (ALF) manifolds. These are manifolds which are asymptotic to a circle fibration over a Euclidean base, with fibres of asymptotically constant length.\par
To introduce such class of metrics, we first consider the total space $\mathcal{E}^{n+1}$ of a principal $S^1$ bundle $\pi$ over $\mathbb{R}^n\setminus \overline{B_R(0)}$. Given a positive number $L$, we introduce the vector field $T$ that is equal to $\frac{L}{2\pi}$ times the infinitesimal generator of the $S^1$ action and consider a ``model metric" $h$ on $\mathcal{E}$, given by
\begin{align}
    h=\pi^* g_{\mathbb{R}^n}+\eta^2
\end{align}
where $\eta$ is a connection $1$-form, namely an $S^1$ invariant $1$-form on $\mathcal{X}$ such that $\eta(T)=1$. Observe that each fiber has length $L$. For such a $1$-form $\eta,$ one might write $d\eta=\pi^*\Omega$ for some $\Omega$ on the base and we will assume this ``curvature" $2$-form  decays at ``infinity", i.e., there is a positive number $\tau$  such that 
\begin{align}\label{curvature 1}
    D^i\Omega= O (r^{-\tau-1-i}),\,0\leq i\leq 2
\end{align}
where $D$ denotes the flat connection on $\mathbb{R}^n$.
\begin{definition}
    An $n+1$- dimensional complete Riemannian manifold $(M,g)$ is said to be \textit{asymptotically locally flat} (ALF) of order $\tau>0$, if there exists a compact subset $K\subset M$ such that $M\setminus K$ has coordinates at infinity: namely there exist $R>0$ and a $C^\infty$-diffeomorphism $\mathcal{X}:M\setminus K\rightarrow \mathcal{E}$ such that there exists a model metric $h$ on $\mathcal{E}$ satisfying \eqref{curvature 1} and 
    \begin{align}
        (\nabla^h)^i\big((\mathcal{X}^{-1})^*g-h\big)=O\big(r^{-\tau-i}),\,0\leq i\leq 3
    \end{align}
\end{definition}
The proofs used in the ALE case extend directly, \textit{mutatis mutandis}, to yield the following theorems in the ALF case.
\begin{theorem}
    Let $(M,g,J)$ be an almost-K\"ahler Einstein ALF manifold of dimension $n+1\geq 4$ and of order $\tau>\frac{2}{3}(n-4)$ with $s\geq 0$ and as $r\longrightarrow\infty$, $|(\nabla^{\mathrm{LC}})^2J|=O (r^{-\eta})$ with $\eta\geq \frac{2}{3}(n-1)$. Then $(M,g,J)$ is K\"ahler Einstein.
\end{theorem}
\begin{theorem}
    Let $(M,g,J)$ be a $4$-dimensional almost-K\"ahler ALF manifold of order $\tau>0$. If $s\geq 0$ and $\delta W_+=0$, then $(M,g,J)$ is a constant-scalar-curvature K\"ahler (cscK) manifold.
\end{theorem}
\begin{theorem}
    Let $(M,g,J)$ be a $4$-dimensional almost-K\"ahler Einstein ALF manifold of order $\tau>0$ with $s\geq 0$. Then $(M,g,J)$ is K\"ahler Einstein.
\end{theorem}
This theorem has an immediate corollary.
\begin{corollary}
\begin{enumerate}
    \item $S^2\times \mathbb{R}^2$ does not admit a symplectic form which is compatible with the Riemannian Kerr metrics.
    \item $\mathbb{CP}^2\setminus S^1$ does not admit a symplectic form which is compatible with the Chen-Teo metric.
    \item $\mathbb{CP}^2\setminus \{$point$\}$ does not admit a symplectic form which is compatible with the Taub-bolt metric.
\end{enumerate}
\end{corollary}
The article is organised as follows: §\ref{section 2} discusses Dirac operators in almost K\"ahler geometry and presents the proof of Theorem \ref{Mass formula}. §\ref{pmt} discusses the positive mass theorem for almost K\"ahler AE manifolds and related results. §\ref{section 3} establishes the rigidity results for almost K\"ahler-Einstein ALE manifolds.\par
\vspace{1 ex}
\textbf{Acknowledgments.} The author is grateful to Olivier Biquard for helpful discussions and for his feedback on an initial draft of this article. He also thanks David Blair for sharing a copy of his article \cite{Blair}. Finally, the author would like to thank the anonymous referee for their valuable suggestions, especially for drawing his attention to the issue discussed in remark \ref{r1}, which significantly helped clarify the scope of this article. This work was supported by European Union’s Horizon 2020 Europe research and innovation programme under the Marie Sklodowska-Curie grant agreement no 101034255.
\subsection{Existence of strictly almost K\"ahler ALE manifolds}
\begin{proof}[Proof of theorem \ref{lemma existence}:]
We start with a K\"ahler ALE manifold $(M^{2m},g_0,J_0)$. 
Choose a point $p\in M$ and a holomorphic coordinate chart $(z_1,\dots,z_m)$ centered at $p$ for the complex structure $J_0$. Let
$e_k=\frac{\partial}{\partial x_k}, 
J_0e_k=\frac{\partial}{\partial y_k}$, where  $z_k=x_k+iy_k.$ Choose a non-constant smooth compactly supported function $f\in C_c^\infty(U)$
supported in the coordinate neighborhood $U$ such that $\partial_{z_1}f(0)\neq 0$ and define a smooth endomorphism $A\in \Gamma(\operatorname{End}(TM))$
by 
\begin{align*}
   Ae_1=f e_1, A(J_0e_1)=-f J_0e_1 \text{ and } Ae_j=A(J_0e_j)=0 \text{ for } j\geq 2 
\end{align*}
Then $AJ_0+J_0A=0$ and $A$ is symmetric with respect to $g_0$, i.e., $g_0(Au,v)=g_0(u,Av)$ because in the orthonormal frame the matrix of $A$ reads diag$(f,-f,0,\cdots,0)$. Now define $J_t=e^{tA}J_0e^{-tA}$. Clearly $J_t^2=-1$. Since $A$ anti-commutes with $J_0$ (this implies $J_t=J_0 e^{-2tA}$) and $A$ is $g_0$-symmetric we have
\begin{align*}
    g_t(J_tu, J_t v)=\omega_0(J_0 e^{-2tA}u,-v)= g_0(e^{-2tA}u,v)=g_0(u,e^{-2tA}v)= \omega_0(u,J_t v)=g_t(u,v)
\end{align*}
For small enough $t, g_t(u,u)>0$ for non-zero $u$. Therefore, $(M,g_t,J_t)$ is almost K\"ahler ALE for small enough $t$. Notice that for small $t$, $J_t$ has an expansion
\begin{align*}
    J_t=J_0+t\,2AJ_0+O(t^2)
\end{align*}
We calculate the Nijenhuis tensor of $J_t$:
\begin{align*}
    N_{J_t}(u,v)=[J_tu,J_tv]-J_t[J_tu,v]-J_t[u,J_tv]-[u,v]
\end{align*}
Using the fact that $N_{J_0}\equiv 0$, we get
\begin{align*}
    N_{J_t}(u,v)&=2t([AJ_0u,J_0v]+[J_0u,AJ_0v]-J_0[AJ_0u,v]-AJ_0[J_0u,v]-J_0[u,AJ_0v]\\
    &\hspace{2 ex}-AJ_0[u,J_0v])+O(t^2)
\end{align*}
In the holomorphic coordinate neighbourhood $U$, we calculate this for $u=\partial_{z_1},v=\partial_{\bar z_1}$. We have 
\begin{align*}
    J_0(\partial_{z_1})=J_0(e_1+J_0 e_1)=-\partial_{\bar z_1}, \,\,J_0(\partial_{\bar z_1})=J_0(e_1-J_0 e_1)=\partial_{z_1}\\
    A(\partial_{z_1})=A(e_1+J_0 e_1)=f\partial_{\bar z_1},\,\, A(\partial_{\bar z_1})=A(e_1-J_0e_1)=f\partial_{z_1}
\end{align*}
Therefore,
\begin{align*}
N_{J_t}(\partial_{z_1},\partial_{\bar z_1})&=2t([-f\partial_{\bar z_1},\partial_{z_1}]+[-\partial_{\bar z_1},f\partial_{z_1}]-J_0[-f\partial_{\bar z_1},\partial_{\bar z_1}]-AJ_0[-\partial_{\bar z_1},\partial_{\bar z_1}]\\
&\hspace{2 ex}-J_0[\partial_{ z_1},f\partial_{ z_1}]-AJ_0[\partial_{z_1},\partial_{z_1}])+O(t^2)\\
&=2t\big((\partial_{z_1}f)\partial_{\bar z_1}-(\partial_{\bar z_1}f)\partial_{ z_1}-(\partial_{\bar z_1}f)\partial_{ z_1}+(\partial_{z_1}f)\partial_{\bar z_1}\big)+O(t^2)\\
&=4t\big((\partial_{z_1}f)\partial_{\bar z_1}-(\partial_{\bar z_1}f)\partial_{ z_1}\big)+O(t^2)
\end{align*}
Extend the vector fields $\partial_{z_1},\partial_{\bar z_1}$ on whole of $M$ smoothly and call them $\partial_{z_1},\partial_{\bar z_1}$ again by abuse of notation. Now define a map
\begin{align*}
    F:t\mapsto N_{J_t}(\partial_{z_1},\partial_{\bar z_1})(p)
\end{align*}
We have
\begin{align*}
    F(0)=0,\,\, F^{'}(0)\neq 0 \text{ since } \partial_{z_1}f(0)\neq 0
\end{align*}
Hence $\exists\epsilon>0$ such that $\forall t\in (0,\epsilon), F(t)\neq 0$. Hence $\forall t\in (0,\epsilon), N_{J_t}\not\equiv 0.$
\end{proof}
Notice if we implement this construction on a K\"ahler ALE manifold, the geometry remains same outside a compact set so the metric remains ALE of the same order. 
\section{The ADM mass and Witten's trick}\label{section 2}
\subsection{Dirac operator in almost K\"ahler geometry}
For any almost Hermitian manifold $(M,g,J)$ of dimension $n=2m\geq 4$, a connection $\nabla$ on $M$ (acting on sections of the tangent bundle $TM$) is \textit{Hermitian} if it preserves the metric $g$ and the almost complex structure $J$:
\begin{align*}
    \nabla g=0,\, \nabla J=0
\end{align*}
Each connection $\nabla$ determines a \textit{Cauchy Riemann operator}, denoted by $\partial^\nabla$, defined as the $(0,1)$-part of $\nabla$:
\begin{align*}
    \bar\partial^\nabla _X Y=\frac{1}{2}(\nabla _X Y+J \nabla_{JX} Y)
\end{align*}
There is also an \textit{intrinsic} Cauchy Riemann operator coming from the almost complex structure $J$ defined by
\begin{align*}
    \bar\partial_X Y=\frac{1}{4}([X,Y]+[JX,JY]+J[JX,Y]-J[X,JY])
\end{align*}
Gauduchon \cite{Gauduchon} distinguished a real affine line in the set of Hermitian connections. We pick one element from the affine line, the \textit{Chern connection}: $\nabla^{\text{Ch}}$. The terminology stems from the fact that, in the integrable case, it coincides with the Chern connection of the tangent bundle viewed as a Hermitian, holomorphic bundle of rank $m$. The Chern connection is characterised by the fact that its Cauchy Riemann operator coincides with the intrinsic one:
\begin{align*}
    \bar\partial^{\nabla^{\text{Ch}}}= \bar\partial
\end{align*}
The metric $g$ defines a Hermitian metric on the anti-cannical bundle $K^{-1}$ and the Chern connection induces a unitary connection on $K^{-1}$, which we call by $A^{\text{Ch}}$.\par 
Now let's talk about spinors. We start with the canonical spinor bundle on $(M,g,J)$:
\begin{align*}
    S_+=\bigoplus_{p=0}^m \Lambda^{0,2p} M,\hspace{4 ex} S_-=\bigoplus_{p=0}^{m-1} \Lambda^{0,2p+1} M
\end{align*}
To define a spin connection on $S=S_+\oplus S_-$, we need a unitary connection $A$ on $K^{-1}$. We call the corresponding Dirac operator $D_A$.\par
There is another intrinsic first order differential operator with the same index of $D_A$, which sends $S_+$ to $S_-$ and vice versa, namely the normalised \textit{Dolbeault} operator defined by $\sqrt{2}(\bar\partial+\bar\partial^*)$. Here $\bar\partial$ on the sections of $\Lambda^{0,p}M$ is defined as 
\begin{align*}
    \bar\partial\psi= (d\psi)^{(0,p+1)}
\end{align*} 
for any $(0,p)$-form $\psi$ and $\bar\partial^*$ is the Hermitian adjoint of $\bar\partial$.\par 
Gauduchon \cite{Gauduchon} proved that on an almost K\"ahler manifold $(M,g,J)$, we have 
\begin{align}\label{Dirac}
    D_{A^\text{Ch}}= \sqrt{2}(\bar\partial+\bar\partial^*)
\end{align}
In dimension $4,$ Taubes \cite{Taubes} obtained the same result by indirect arguments and it is a key fact in his \textit{magnum opus}: ``SW$=$Gr''.\par 
For any unitary connection $A$ on $K^{-1}$ and a spinor $\psi\in \Gamma(S)$, the gauge group $\mathcal{G}=$ Map$(M,S^1)$ acts on the pair $(A,\psi)$ by pull-back. Hence for $g\in\mathcal{G},$
\begin{align*}
    g\cdot (A,\psi)=(A-2g^{-1}dg, g\psi)
\end{align*}
Moreover, $(A,\psi)$ solves the Dirac equation $D_A\psi=0$ iff $g\cdot (A,\psi)$ solves it.\par 
The discussion above along with the identity \eqref{Dirac} gives us the following lemma.
\begin{lemma}\label{lemma 1}
    Up to gauge equivalence there exists a canonical connection $A_0$ and a non-trivial spinor $\varphi_0\in\Omega^0 \subset \Gamma(S_+)$, such that $D_{A_0}\varphi_0=0$. $\varphi_0$ can be taken to have point-wise norm one.
\end{lemma}
\begin{proof}
    Take $\psi_0\in\Omega^0$ to be to the constant section of $M\times\mathbb{C}$ which assigns $1\in\mathbb{C}$ to every point of $M$. Then $A_0=A^{\text{Ch}},\varphi_0=\psi_0$ solves the Dirac equation $D_{A_0}\varphi_0=0$ because $d\psi_0=0$ and therefore $\bar\partial\psi_0=0$.
\end{proof}
In spinorial terms the Chern connection is characterised as follows (see §1.4.3 in \cite{Nicol} for a detailed exposition). For the constant section $\psi_0,$ $A^{\text{Ch}}$ is the unique connection on $K^{-1}$ such that the induced connection on $S,\nabla_{A^\text{Ch}}$ satisfies 
\begin{align*}
    \nabla_{A^\text{Ch}}\, \psi_0\in\Gamma(\Lambda^1\otimes \Lambda^{0,2})
\end{align*}
In fact just using this condition above one can prove that $D_{A^{\text{Ch}}}\,\psi_0=0$. To see this first notice that $\Omega^0$ is a $-mi$ eigenspace of the Clifford action of the symplectic form: cl$(\omega)$ and $\Omega^{0,2}$ is a $ci$ eigenspace of cl$(\omega)$, where $c\neq -m$. For example, $c=2$ when $m=2$, $c=1$ when $m=3$, and $c=0$ when $m=4$. Now apply $D_{A^{\text{Ch}}}$ on cl$(\omega)\psi_0$ to get 
\begin{align}\label{eq 1}
    D_{A^{\text{Ch}}}\big(\text{cl}(\omega)\psi_0\big)= \text{cl}\big(d\omega+d^*\omega\big)\psi_0+\text{cl}\big((1\otimes \text{cl}(\omega))\,\nabla_{A^{\text{Ch}}} \psi_0\big)
\end{align}
The rightmost term here means the following: start with $\nabla_{A^{\text{Ch}}} \psi_0\in\Gamma(\Lambda^1\otimes\Lambda^{0,2})$, first apply $1\otimes \text{cl}(\omega)$ to get another element of $\Gamma(\Lambda^1\otimes\Lambda^{0,2})$ and then apply cl to get an element of $\Gamma(S_-)$. Because $\omega$ is closed and co-closed, equation \eqref{eq 1} becomes
\begin{align*}
    -mi D_{A^{\text{Ch}}}(\psi_0)=ci D_{A^{\text{Ch}}}(\psi_0)
\end{align*}
Since $c\neq -m$, this proves $D_{A^{\text{Ch}}}\,\psi_0=0$.\par
$\nabla_{A^{\text{Ch}}}\,\psi_0=0$ iff $(M,g,J)$ is K\"ahler, instead if it is almost K\"ahler, we still have $D_{A^{\text{Ch}}}\,\psi_0=0$. We also have $\nabla^{\mathrm{LC}}\,\omega\in\Gamma\big(\Lambda^1\otimes (\Lambda^{2,0}\oplus \Lambda^{0,2})\big)$ and in fact $\nabla_{A^{\text{Ch}}}\,\psi_0$ and $\nabla^{\mathrm{LC}}\,\omega$ are essentially the same. In concrete terms, we have a pointwise equality of norms: 
\begin{lemma}\label{Norm equality}
\begin{align}\label{norm equality}
    |\nabla_{A^{\mathrm{Ch}}}\,\psi_0|^2=\frac{1}{8}|\nabla^{\mathrm{LC}}\,\omega|^2
\end{align}
\end{lemma}
\begin{proof}
    The Weitzenböck formula gives us
    \begin{align*}
        0=D^2_{A^{\text{Ch}}}\,\psi_0=\nabla^*_{A^{\text{Ch}}}\nabla_{A^{\text{Ch}}}\,\psi_0+\frac{s}{4}\psi_0+\frac{1}{2}\text{cl}(F_{A^{\text{Ch}}})\,\psi_0
    \end{align*}
    Now take inner product with $\psi_0$ on both sides of the equation. From Blair's calculation \cite{Blair}, it is evident that the only term in
$\mathrm{cl}(F_{A^{\mathrm{Ch}}})$ which contributes to the inner product is $-\frac{i}{4m}(s + s^*)\,\mathrm{cl}(\omega)$. All remaining curvature terms lie in directions orthogonal to $\omega$ and map $\Omega^0$ into forms orthogonal to $\Omega^0$.
Thereafter we have
    \begin{align}\label{eq 2}
        0=\big\langle \nabla^*_{A^{\text{Ch}}}\nabla_{A^{\text{Ch}}}\,\psi_0,\psi_0\big\rangle+\frac{s}{4}|\psi_0|^2-\frac{i}{8m}(s+s^*)\big\langle\text{cl}(\omega)\psi_0,\psi_0\big\rangle
    \end{align}
    Since $\psi_0$ has pointwise constant norm $1$,
    \begin{align*}
        \big\langle \nabla^*_{A^{\text{Ch}}}\nabla_{A^{\text{Ch}}}\,\psi_0,\psi_0\big\rangle=|\nabla_{A^{\text{Ch}}}\,\psi_0|^2+\frac{1}{2}\Delta |\psi_0|^2=|\nabla_{A^{\text{Ch}}}\,\psi_0|^2
    \end{align*}
    $\Omega^0$ is a $-mi$ eigenspace of cl$(\omega)$. Therefore equation \eqref{eq 2} reads
    \begin{align*}
        0=|\nabla_{A^{\text{Ch}}}\,\psi_0|^2+\frac{s}{4}-\frac{2s+|\nabla^{\text{LC}}\,\omega|^2}{8}
    \end{align*}
    which in turns gives $|\nabla_{A^{\text{Ch}}}\,\psi_0|^2=\frac{1}{8}|\nabla^{\mathrm{LC}}\,\omega|^2$.
\end{proof}
Blair's article \cite{Blair} is difficult to find, instead one can also look at his book \cite{Blair2} (proof of theorem 10.6) for the calculation used above. In dimension $4$, LeBrun gave a different proof of this calculation using Penrose spinor calculus (\cite{Penrose}, proposition 2.1).  
\subsection{The mass formula}
We adapt Witten's spectacular observation on computing the ADM mass $m(M,g)$ of a spin manifold in terms of a harmonic spinor. To set up the proof of the derivation of the mass formula \eqref{mass formula}, we need a lemma and a proposition.\par
We take a frame $(e_1,\dots,e_n)$, parallel at an arbitrary point $p\in M.$ The dual coframe is given by $(e^1,\dots,e^n)$.
\begin{lemma}\label{lemma 1}
    Let $N\subset M$ be any bounded subset of $M$ with $\partial N$ smooth. Then for any $\varphi\in\Gamma(S),$ we have
    \begin{align*}
        \int_N\big(|\nabla_A\varphi|^2-|D_A\varphi|^2+\frac{s}{4}|\varphi|^2+\frac{1}{2}\langle \mathrm{cl}(F_A)\varphi,\varphi\rangle\big)\hspace{0.5 ex}d\mathrm{v}_g=\int_{\partial N} \sum_{i=1}^n \langle\varphi,L_i\varphi\rangle * e^i
    \end{align*}
    where $L_i=\sum_{j=1}^n\big(\delta_{ij}+\mathrm{cl}(e_i)  \mathrm{cl}(e_j)\big)\nabla_j$. 
\end{lemma}
\begin{proof} We make the following calculations.
    \begin{align*}
        -|D_A\varphi|^2&=\sum_{i,j=1}^n\langle-\mathrm{cl}(e_i)\nabla_i\varphi,\mathrm{cl}(e_j)\nabla_j\varphi\rangle\\
        &=\sum_{i,j=1}^n\langle\nabla_i\varphi,\mathrm{cl}(e_i)  \mathrm{cl}(e_j)\nabla_j\varphi\rangle\\
        &=\sum_{i,j=1}^n\nabla_i\langle\varphi,\mathrm{cl}(e_i)  \mathrm{cl}(e_j)\nabla_j\varphi\rangle-\langle\varphi,D_A^2\varphi\rangle
    \end{align*}
    and 
    \begin{align*}
        |\nabla_A\varphi|^2&=\sum_{i=1}^n\langle\nabla_i\varphi,\nabla_i\varphi\rangle\\
        &=\sum_{i=1}^n\nabla_i\langle\varphi,\nabla_i\varphi\rangle-\sum_{i=1}^n\langle\varphi,\nabla_i\nabla_i\varphi\rangle\\
        &=\sum_{i=1}^n\nabla_i\langle\varphi,\nabla_i\varphi\rangle+\langle\varphi,\nabla_A^*\nabla_A\varphi\rangle
    \end{align*}
Finally we complement these calculations above with the Weitzenb\"ock formula
\begin{align*}
    D_A^2=\nabla_A^*\nabla_A+\frac{s}{4}+\frac{1}{2}\mathrm{cl}(F_A)
\end{align*}
and the divergence theorem to finish the proof.
\end{proof}
\begin{proposition}\label{proposition 1}
    Let $(M,g,J)$ be an ALE almost K\"ahler manifold with $\tau>\frac{n-2}{2}$ and $s,|\nabla^{\mathrm{LC}}J|^2\in L^1$. Then in any asymptotic coordinate system, the ADM mass is given by 
    \begin{align*}
        \mathfrak{m}(M,g)=\lim_{r\rightarrow \infty } \frac{1}{2(2m-1)\pi^m}\int_{S_r/\Gamma}\theta\wedge\omega^{m-1} 
    \end{align*}
    for any $1$-form $\theta$ with $d\theta=iF_{A^{\mathrm{Ch}}}$ on the end $M_\infty$. 
\end{proposition}
\begin{proof}
    Taking the Bartnik-Chru\'sciel coordinate invariance of the mass \cite{Bartnik, Chruciel}, we begin by first checking that the assertion is true in a particular asymptotic coordinate system and for a particular choice of $\theta$.\par 
    We set up the calculations in local coordinates near ``infinity". Let $U\subset M\rightarrow Q$ be a local section of the frame bundle $Q.$ $e$ consists of an orthonormal frame $e=(e_1,\dots,e_n)$ of vector fields defined on the open set $U\subset M.$ The local connection form $Z^e=e^*(Z):TU\rightarrow \mathfrak{so}(n)$ is given by the formula:
    \begin{align*}
        Z^e=\sum_{k<l}w_{kl}E_{kl}
    \end{align*}
    where the forms $w_{kl}$ denote the forms defining the Levi-Civita connection, $w_{kl}=g(\nabla e_k,e_l)$ and $E_{kl}$ are the standard basis matrices of the Lie algebra $\mathfrak{so}(n).$ Analogously, we fix a section $s:U\rightarrow P_1$ of the $U(1)$-principal bundle associated to the canonical spin$^\mathbb{C}$-structure and obtain a local connection form 
    \begin{align*}
        A_s=s^*(A):TU\rightarrow i\mathbb{R}
    \end{align*}
    $A_s$ is an imaginary valued one form defined on the set $U.$ A section $\varphi\in\Gamma(U,S)$ of the spinor bundle $S$ over $U$ can be described by a function $\varphi:U\rightarrow\mathbb{C}^{\big(2^{\frac{n}{2}}\big)}$ using a trivialisation of $S$ over $U$ and using the connection $A$ on the corresponding $U(1)$-bundle, we lift the Levi-Civita connection on $S.$ Finally the covariant derivative of $\varphi$ is computed according to the formula \cite{Friedrich}:
    \begin{align}\label{covariant deribative of a spinor}
        \nabla_A\varphi=d\varphi+\frac{1}{2}\sum_{k<l}w_{kl}\,\mathrm{cl}(e_k)\mathrm{cl}(e_l)\varphi+\frac{1}{2}A_s\varphi
    \end{align}
    Now we take $\varphi=\psi_0,A=A^{\mathrm{Ch}}$ and a bounded subset $N\subset M$ with $\partial N=S_r/\Gamma$, lemma \ref{Norm equality} and \ref{lemma 1} together give us
    \begin{align}
        \int_{S_r/\Gamma} \sum_{i=1}^n \langle\psi_0,L_i\psi_0\rangle *e^i=0
    \end{align}
    Using \eqref{covariant deribative of a spinor} we compute the integrand. First we take care of the part which involves the connection $A^{\mathrm{Ch}}$.
    \begin{align*}
        &\sum_{i,j=1}^n \big\langle \psi_0,\big(\delta_{ij}+\mathrm{cl}(e_i)\mathrm{cl}(e_j)\big)A_s^{\mathrm{Ch}}(e_j)\psi_0\big\rangle* e^i\\
        &=\sum_{i,j=1}^n \big\langle \psi_0,\mathrm{cl}(e_i)\mathrm{cl}(e_j)A_s^{\mathrm{Ch}}(e_j)\psi_0\big\rangle* e^i+\sum_{i=1}^m\big\langle\psi_0,A_s^{\mathrm{Ch}}(e_i)\psi_0\big\rangle* e^i\\ 
        &=\sum_{i=1}^n\big\langle\psi_0,\mathrm{cl}(e_i)\mathrm{cl}(A_s^{\mathrm{Ch}})\psi_0\big\rangle *e^i +\sum_{i=1}^m\big\langle\psi_0,A_s^{\mathrm{Ch}}(e_i)\psi_0\big\rangle* e^i
    \end{align*}
We use explicit formula for Clifford action of a $1$-form to simplify it further. For a complexified $1$-form $\alpha\in\Omega^1\otimes\mathbb{C}$ and a spinor $\xi\in\Omega^{0,k}$, the Clifford action of $\alpha=(\alpha^{0,1}+\alpha^{1,0})$ on $\xi$ is given by the following formula:
\begin{align}\label{Clifford one form}
\mathrm{cl}(\alpha)\xi=\sqrt{2}\big(\alpha^{0,1}\wedge\xi-\overline{\alpha^{1,0}}\righthalfcup\xi\big) 
\end{align}
$\overline{\alpha^{1,0}}\righthalfcup\xi\in\Omega^{0,k-1}$ is the contraction of $\xi$ by $(\overline{\alpha^{1,0}})^{\#}$ defined by $\overline{\alpha^{1,0}}\righthalfcup\xi:=\iota_{(\overline{\alpha^{1,0}})^{\#}}(\xi)$, where $(\overline{\alpha^{1,0}})^{\#}$ is the vector field metric dual to $\overline{\alpha^{1,0}}$. Using this we notice
\begin{align*}
    &\big\langle\psi_0,\mathrm{cl}(e_i)\mathrm{cl}(A_s^{\mathrm{Ch}})\psi_0\big\rangle +\big\langle\psi_0,A_s^{\mathrm{Ch}}(e_i)\psi_0\big\rangle \\
    &=\big\langle\psi_0,\mathrm{cl}(e_i) \big(\sqrt{2}\,\psi_0 (A_s^{\mathrm{Ch}})^{0,1} \big)+A_s^{\mathrm{Ch}}(e_i)\psi_0\big\rangle\\ 
    &=\big\langle \psi_0,-2\psi_0\,e_i^{0,1}\righthalfcup(A_s^{\mathrm{Ch}})^{0,1}+A_s^{\mathrm{Ch}}(e_i)\psi_0\big\rangle\\
    &=-2\big\langle\psi_0,\psi_0 (A_s^{\mathrm{Ch}})^{0,1}(e_i^{0,1})\big\rangle+A_s^{\mathrm{Ch}}(e_i)\\
    &=-2\big\langle\psi_0, \Big(\frac{1}{2}(A_s^{\mathrm{Ch}}+i A_s^{\mathrm{Ch}}\circ J)\Big)\Big(\frac{1}{2}(e_i+i Je_i)\Big)\psi_0+A_s^{\mathrm{Ch}}(e_i)\psi_0\big\rangle\\ 
    &=\big\langle\psi_0,\big(-A_s^{\mathrm{Ch}}(e_i)-iA_s^{\mathrm{Ch}}(Je_i)+A_s^{\mathrm{Ch}}(e_i)\big)\psi_0\big\rangle\\
    &=-iA_s^{\mathrm{Ch}}(Je_i)
\end{align*}
One can infact choose an orthonormal basis at $p\in M$ of the form $\{e_i\}=\{e_k,Je_k\}$. Then we get $\omega=\sum_k e^k\wedge Je^k$ and 
\begin{align*}
    \sum_{i=1}^n iA_s^{\mathrm{Ch}}(Je_i)*e_i= iA_s^{\mathrm{Ch}}\wedge\frac{\omega^{m-1}}{(m-1)!}
\end{align*}
Now the total integrand reads
    \begin{align*}
        &\sum_{i=1}^n\langle \psi_0,L_i\psi_0\rangle* e^i\\
        &=\sum_{i,j=1}^n \big\langle\psi_0,\big(\delta_{ij}+\mathrm{cl}(e_i)\mathrm{cl}(e_j)\big)\nabla^{\mathrm{Ch}}_j\psi_0\big\rangle* e^i\\
        &=\frac{1}{4}\sum_{k\neq l}\sum_{i,j=1}^n\big\langle\psi_0,\big(\delta_{ij}+\mathrm{cl}(e_i)\mathrm{cl}(e_j)\big)w_{kl}(e_j)\mathrm{cl}(e_k)\mathrm{cl}(e_l)\psi_0\big\rangle* e^i-\frac{i}{2}A_s^{\mathrm{Ch}}\wedge\frac{\omega^{m-1}}{(m-1)!}\\
        &=\frac{1}{4}\sum_{k\neq l}\sum_{i\neq j}w_{kl}(e_j)\big\langle\psi_0,\mathrm{cl}(e_i)\mathrm{cl}(e_j)\mathrm{cl}(e_k)\mathrm{cl}(e_l)\psi_0\big\rangle* e^i-\frac{i}{2}A_s^{\mathrm{Ch}}\wedge\frac{\omega^{m-1}}{(m-1)!}\\
        &=\frac{1}{4}\sum_{k\neq l}\sum_{i\neq j}w_{kl}(e_j)|\psi_0|^2(\delta_{il}\delta_{jk}-\delta_{ik}\delta_{jl})* e^i-\frac{i}{2}A_s^{\mathrm{Ch}}\wedge\frac{\omega^{m-1}}{(m-1)!}\\
        &=\frac{1}{4}\sum_{i\neq j} \big(\omega_{ji}(e_j)-\omega_{ij}(e_j)\big)* e^i-\frac{i}{2}A_s^{\mathrm{Ch}}\wedge\frac{\omega^{m-1}}{(m-1)!}
    \end{align*}
Since the Christoffel symbols $\bm{\Gamma}_{ij}^k=g(\nabla_{e_{i}}e_j,e_k)$ are symmetric in $i\leftrightarrow j,$ we get $w_{jk}(e_i)=w_{ik}(e_j).$ We notice that
\begin{align*}
    \omega_{ji}(e_j)-\omega_{ij}(e_j)=\omega_{ji}(e_j)-\omega_{jj}(e_i)
\end{align*}
On the other hand the expression $(\partial_j g_{ji}-\partial_i g_{jj})$ also reads
\begin{align*}
    (\nabla_j g_{ji}-\nabla_i g_{jj}) &=g(\nabla_{e_j}e_j,e_i)+g(e_j,\nabla_j e_i)-g(\nabla_i e_j,e_j)-g(e_j,\nabla_i e_j)\\
    &=w_{ji}(e_j)+w_{ij}(e_j)-2w_{jj}(e_i)\\
    &=w_{ji}(e_j)+w_{jj}(e_i)-2w_{jj}(e_i)\\
    &=w_{ji}(e_j)-w_{jj}(e_i)
\end{align*}
Therefore we get
\begin{align}\label{eq 35}
    \sum_{i=1}^n\langle \psi_0,L_i\psi_0\rangle* e^i
    =\frac{1}{4}\sum_{i\neq j}(\nabla_j g_{ji}-\nabla_i g_{jj})* e^i-\frac{i}{2}A_s^{\mathrm{Ch}}\wedge\omega^{m-1}
\end{align}  
Bartnik \cite{Bartnik} showed us that 
\begin{align*}
    \lim_{r\rightarrow \infty}\int_{S_r/\Gamma}\sum_{i\neq j}(\nabla_j g_{ji}-\nabla_i g_{jj})* e^i=\int_{S_\infty/\Gamma} \sum\big(\partial_i(\phi^*g)_{ij}(x)-\partial_j(\phi^*g)_{ii}(x)\big)\partial_j \righthalfcup dx
\end{align*}
Using \eqref{eq 35} we therefore have
\begin{align*}
    \mathfrak{m}(M,g)=\lim_{r\rightarrow \infty} \frac{1}{2(2m-1)\pi^m}\int_{S_r/\Gamma}\theta\wedge\omega^{m-1}
\end{align*}
for a particular $1$-form $\theta=A_s^{\mathrm{Ch}}$ with $d\theta=iF^{\mathrm{Ch}}$ on the end $M_{\infty}$. On the other hand, since $b_1(M_\infty)=0$, any other $1$-form $\tilde\theta$ on $M_\infty$ with $d\tilde\theta=iF_{A^{\mathrm{Ch}}}$ is given by $\tilde\theta=\theta+df$ for a function $f$.  Choosing a different $\theta$ would thus change the integrand by an exact form, leaving the integral on each $S_r/\Gamma$ completely unchanged.\par 
Finally, the limit is independent of the asymptotic coordinate system. Notice that
\begin{align*}
    d(\theta\wedge\omega^{m-1})=iF_{A^{\mathrm{Ch}}}\wedge\omega^{m-1}=\frac{s+s^{*}}{4m}\omega^m=\frac{(m-1)!}{4}(s+s^*)d\mathrm{v}_g
\end{align*}
Thereafter if $N_1$ is a real hypersurface in the region $E_r$ exterior to $S_r/\Gamma$ such that $S_r/\Gamma$ and $N_1$ are the boundary components of a bounded region $V\subset E_r$, then 
\begin{align*}
    &\frac{4}{(m-1)!}\bigg|\int_{N_1}\theta\wedge\omega^{m-1}-\int_{S_r/\Gamma}\theta\wedge\omega^{m-1}\bigg|\\&=\bigg|\int_V (s+s^{*})d\mathrm{v}_g\bigg|\\ 
    &\leq \int_V |s+s^{*}|d\mathrm{v}_g\\&\leq \int_{E_r} |s+s^{*}|d\mathrm{v}_g
    \longrightarrow 0 \text{ as }r\longrightarrow\infty
\end{align*}
since by hypothesis $(s+s^{*})\in L^1$.
\end{proof}
\begin{proof}[Proof of theorem \ref{Mass formula}]
Choose a $1$-form $\theta$ defined on $M_\infty$ such that
\begin{align*}
d\theta = iF_{A^{\mathrm{Ch}}}
\end{align*}
Next choose asymptotic coordinates on $M_\infty$, and let $r$ be  corresponding coordinate radius at infinity. Next, choose a smooth cut--off function $f : M \rightarrow [0,1]$ which vanishes identically on the compact region $M \setminus M_\infty$ and satisfies $f \equiv 1$ for $r \ge R_0$, where $R_0$ is a sufficiently large real number. Define
\begin{align*}
\xi := iF_{A^{\mathrm{Ch}}} - d(f\theta)
\end{align*}
Hence $\xi$ is a closed $2$-form with compact support. Moreover, $\xi$ is cohomologous to $iF_{A^{\mathrm{Ch}}}$ and hence represents
the class $2\pi\,\iota^{-1}(c_1)=\iota^{-1}(iF_{A^{\mathrm{Ch}}})$ in the
compactly supported cohomology $H^2_c(M)$.\par 
For $\tilde R > R_0$, let $M_{\tilde R}$ be the compact manifold with boundary obtained by truncating $M$ at radius $\tilde R$, so that $\partial M_{\tilde R} = S_{\tilde R}/\Gamma$. Using the identity \cite{Blair}
\begin{align*}
iF_{A^{\mathrm{Ch}}} \wedge \omega^{m-1}
= \frac{s+s^{*}}{4m}\,\omega^m
= \frac{(m-1)!}{4}(s+s^{*})\, d\mathrm{v}_g
\end{align*}
we obtain
\begin{align*}
\frac{(m-1)!}{4}\int_{M_{\tilde R}} (s+s^{*})\, d\mathrm{v}_g
= \int_{M_{\tilde R}} iF_{A^{\mathrm{Ch}}}\wedge \omega^{m-1}
= \int_{M_{\tilde R}} \big(\xi + d(f\theta)\big)\wedge \omega^{m-1}
\end{align*}
It follows that
\begin{align*}
2\pi \,\langle \iota^{-1}(c_1),[\omega]^{m-1}\rangle
&= \int_M \xi \wedge \omega^{m-1} \\
&= \int_{M_{\tilde R}} \xi \wedge \omega^{m-1} \\
&= - \int_{M_{\tilde R}} d(f\theta \wedge \omega^{m-1})
   + \frac{(m-1)!}{4}\int_{M_{\tilde R}} (s+s^{*})\, d\mathrm{v}_g \\
&= - \int_{\partial M_{\tilde R}} f\theta \wedge \omega^{m-1}
   + \frac{(m-1)!}{4}\int_{M_{\tilde R}} (s+s^{*})\, d\mathrm{v}_g \\
&= - \int_{S_{\tilde R}/\Gamma} \theta \wedge \omega^{m-1}
   + \frac{(m-1)!}{4}\int_{M_{\tilde R}} (s+s^{*})\, d\mathrm{v}_g 
\end{align*}
Rearranging terms yields
\begin{align*}
\frac{1}{2(2m-1)\pi^m}
\int_{S_{\tilde R}/\Gamma} \theta \wedge \omega^{m-1}
=
-\frac{\langle \iota^{-1}(c_1),[\omega]^{m-1}\rangle}{(2m-1)\pi^{m-1}}
+\frac{(m-1)!}{8(2m-1)\pi^m}
\int_{M_{\tilde R}} (s+s^{*})\, d\mathrm{v}_g 
\end{align*}
Finally, taking the limit as $\tilde R \to \infty$ and invoking the mass
formula from proposition \ref{proposition 1}, we obtain
\begin{align*}
\mathfrak{m}(M,g)
=
-\frac{\langle \iota^{-1}(c_1),[\omega]^{m-1}\rangle}{(2m-1)\pi^{m-1}}
+\frac{(m-1)!}{8(2m-1)\pi^m}
\int_M (s+s^{*})\, d\mathrm{v}_g
\end{align*}
which completes the proof.
\end{proof}
\section{Positive mass theorem and Penrose inequality in dimension $4$}\label{pmt}
\subsection{Capping off the end with LeBrun's $\Gamma$-capsule}
The proofs of theorems~\ref{16}--\ref{PMT} follow the strategy developed by LeBrun~\cite{LeBrun100}, with some modifications to account for the almost K\"ahler (rather than K\"ahler) setting. LeBrun's strategy relies on a natural symplectic compactification of the ALE symplectic $4$-manifold $(M,\omega)$, producing a closed symplectic $4$-manifold $(\widehat M,\widehat\omega)$ that contains $M$ as the complement of a compact symplectic domain. This construction makes it possible to invoke deep results from $4$-dimensional symplectic topology. We start by defining LeBrun's \textit{$\Gamma$-capsules}, which are needed to cap off the ALE symplectic manifolds $(M,\omega)$.\par 
For any finite subgroup $\Gamma\subset U(2)$ ($\Gamma\neq\{1\}$) that acts freely on the unit sphere $S^3\subset\mathbb{C}^2$, LeBrun \cite{LeBrun100} constructed $4$-dimensional compact connected symplectic orbifolds $(X_\Gamma,\omega_\Gamma)$ such that
\begin{enumerate}[label=(\roman*)]
\item $(X_\Gamma,\omega_\Gamma)$ contains exactly one singular point $p$.\\
\item $p$ has a neighbourhood symplectomorphic to $(B,\omega_0)/\Gamma$ for some standard open ball $B\subset\mathbb{C}^2$ centered at the origin, where $\Gamma$ acts on $(\mathbb{C}^2,\omega_0)$ in the tautological manner, as a subgroup of $U(2), \omega_0=dx^1\wedge dx^2+dx^3\wedge dx^4$ is the standard symplectic form $\mathbb{R}^4=\mathbb{C}^2$.\\
\item\label{3} There is a symplectic immersion $j:S^2 \looparrowright X_\Gamma\setminus \{p\}$, with at worst transverse positively oriented double points, such that 
\begin{align*}
    \int_{S^2} j^*[c_1(X_\Gamma\setminus\{p\},J]\geq 3
\end{align*}
for some, and hence any $\omega$-compatible almost complex structure $J$.
\end{enumerate}
\begin{definition}
Let $\Gamma\subset U(2)$ be any finite subgroup that acts freely on the unit sphere $S^3\subset\mathbb{C}^2$. Then
\begin{itemize}
    \item If $\Gamma\neq\{1\},$ a \textit{$\Gamma$-capsule} will mean one of the standard symplectic orbifolds $(X_\Gamma,\omega_\Gamma)$ satisfying conditions (1)--(3) that LeBrun constructed in §2 of \cite{LeBrun100}. The unique singular point $p\in X_\Gamma$ will then be called the base-point of the $\Gamma$-capsule.
    \item When $\Gamma=\{1\},$ we instead define the associated $\Gamma$-capsule $(X_\Gamma,\omega_\Gamma)$ to be $\mathbb{CP}^2$, equipped with its standard Fubini-Study symplectic structure. In this case the base-point $p$ of $X_\Gamma$ will simply mean $[0:0:1]\in\mathbb{CP}^2$.
    \end{itemize}
\end{definition}
The next proposition constitutes the first step in the symplectic compactification procedure. Its role is to show that, under suitable decay assumptions, the asymptotic symplectic structure of an ALE almost K\"ahler manifold can be normalized near infinity and identified with the standard symplectic structure on $\mathbb{R}^4$. This normalisation is essential, as it allows one to cap off the noncompact end by a compact symplectic domain in a canonical way. Before going into the proposition, we discuss the asymptotic symplectic structure of an ALE $4$-manifold.\par
Let $(M^4,g,J)$ be an ALE K\"ahler surface satisfying the hypothesis of theorem \ref{Penrose}, i.e., it has order $\tau>1$ and as $r\longrightarrow\infty, \nabla^{\mathrm{LC}}J=O(r^{-(2+\epsilon_2)})$ with $\epsilon_2>0$. We assume that there are $(x_1,\dots,x_4)$ asymptotic coordinates on the universal cover $\widetilde{M}_\infty (\cong \mathbb{R}^4\setminus \overline{B_R(0)})$ of $M_\infty$ in which the components of the metric satisfy 
\begin{align*}
    g_{jk}=\delta_{jk}+O\big(r^{-(1+\epsilon_1)}\big),\hspace{2 ex} |\partial_l\,g_{jk}|=O\big(r^{-(2+\epsilon_1)}\big)
\end{align*}
for some $\epsilon_1>0$, and such that the fundamental group $\Gamma$ of the end acts by rotations of the coordinates $(x_1,\dots,x_4)$ in a manner that preserves both the background model Euclidean metric $\delta$ and $g$.\par 
Define $\epsilon:=$ min$(\epsilon_1,\epsilon_2)$ and let $\nabla^\delta$ be the Levi-Civita connection of the Euclidean metric $\delta$. Our fall-off hypothesis implies that as $r\longrightarrow\infty$,
\begin{align*}
    \nabla^{\mathrm{LC}}=\nabla^\delta+\bm{\Gamma}^k_{ij}=\nabla^\delta+O \big(r^{-(2+\epsilon)}\big)
\end{align*}
Since $\nabla^{\mathrm{LC}}J=O(r^{-(2+\epsilon_2)}),\epsilon_2>\epsilon>0$, we have $\nabla^\delta J=O\big(r^{-(2+\epsilon)}\big)$. Now we identify all the tangent spaces of $\mathbb{R}^4$ in the usual way, using the flat Euclidean connection $\nabla^\delta$. Since $\nabla^\delta J=O\big(r^{-(2+\epsilon)}\big)$, the value of $J$ will approach a well-defined limit $J_0$ along some chosen radial ray, and we then extend this limit as a constant-coefficient tensor field on our asymptotic coordinate domain. Along the chosen ray, we then have $J-J_0=O\big(r^{-(1+\epsilon)}\big)$ and integrating along great circles in spheres of constant radius then shows that $J-J_0$ has $O\big(r^{-(1+\epsilon)}\big)$ fall-off everywhere. The same argument similarly also shows that the fall-off for derivative of $J:\nabla^\delta J$ is of the order $O\big(r^{-(2+\epsilon)}\big)$. All together we have the following: \textit{there is a $\delta$-compatible constant-coefficient almost-complex structure $J_0$ on $\mathbb{R}^4$ such that}
\begin{align*}
    J=J_0+O\big(r^{-(1+\epsilon)}\big),\hspace{2 ex} \nabla^\delta J=O\big(r^{-(2+\epsilon)}\big)
\end{align*}
Moreover after rotating our coordinates $(x_1,\cdots,x_4)$ if necessary, we may arrange $J_0$ to become the standard complex structure
\begin{align*}
    dx_1\otimes \frac{\partial}{\partial x_2}-dx_2\otimes \frac{\partial}{\partial x_1}+dx_3\otimes\frac{\partial}{\partial x_4}-dx_4\otimes\frac{\partial}{\partial x_3}
\end{align*}
on $\mathbb{C}^2= \mathbb{R}^4$. The argument given above is borrowed from §2 of \cite{HeinLebrun}.\par 
Since the action of the fundamental group $\Gamma$ preserves both $J$ and $\delta$, it now follows that $\Gamma\subset U(2)$. More importantly, we therefore automatically obtain fall-off conditions
\begin{align}\label{falloff}
    \omega=\omega_0+O\big(r^{-(1+\epsilon)}\big),\hspace{2 ex} \nabla^\delta\omega=O\big(r^{-(2+\epsilon)}\big)
\end{align}
where $\omega_0=dx_1\wedge dx_2+ dx_3\wedge dx_4$ is the standard symplectic form on $\mathbb{R}^4=\mathbb{C}^2$. Now we can state the proposition, which is essentially due to LeBrun \cite{LeBrun100}.
\begin{proposition}\label{prop1}
    Let $(M,g,J)$ be an ALE almost K\"ahler surface, let $\widetilde{M_\infty}$ be the end of the universal cover of $M_\infty$, the end of $M$ and let
    \begin{align*}
        (x_1,x_2,x_3,x_4):\widetilde{M_\infty}\rightarrow\mathbb{R}^4\setminus B
    \end{align*}
    be a diffeomorphism, where $B\subset\mathbb{R}^4$ is a standard closed ball of some large radius centered at the origin. Moreover, suppose that these asymptotics coordinates have been chosen in accordance with the above discussion, so that the K\"ahler form $\omega$ satisfies the fall-off conditions \eqref{falloff} in this coordinate system, while the action of $\pi_1(M_\infty)$ on $\widetilde{M_\infty}$ by deck transformations is represented in these coordinates by the action of a finite group $\Gamma\subset U(2)$ of unitary transformations, acting on $\mathbb{R}^4=\mathbb{C}^2$ in the usual way. Then there is a $\Gamma$ equivariant $C^2$-diffeomorphism $\Phi:\mathbb{R}^4\setminus C\rightarrow \mathbb{R}^4\setminus D$, where $C\subset \mathbb{R}^4$ is a standard closed ball centered at the origin, where $D\subset \mathbb{R}^4$ is a smooth $4$-ball whose boundary $\partial D$ is a $\Gamma$-invariant differentiable $S^3$, and where $B\subset C\cap D$, such that
    \begin{align*}
        \Phi^*\omega=\omega_0
    \end{align*}
    with $|\Phi(x)-x|=O\big(r^{-\epsilon}\big)$ and $|\Phi_*-I|=O\big(r^{-(1+\epsilon)}\big)$.
\end{proposition} 
\begin{proof}
    This proposition is same as proposition~1.1 (\S 1) in \cite{LeBrun100}, with the K\"ahler condition replaced by the almost K\"ahler assumption. LeBrun’s proof of his proposition~1.1 does not rely on the integrability of $J$, but only on the fall-off conditions \eqref{falloff}. Consequently, his argument carries over verbatim to the present setting.
\end{proof}
Now we are in a position to explain the capping-off procedure to get a closed symplectic manifold $(\widehat M,\widehat \omega)$ from $(M,\omega)$. This is explained in the proof of theorem~3.1 in \cite{LeBrun100}. We include it here for completeness.\par
Say $M_\infty\cong (S^3/\Gamma)\times\mathbb{R}^+$. Choose a $\Gamma$-capsule $(X_\Gamma,\omega_\Gamma)$. Proposition \ref{prop1} tells us that the end $M_\infty$ contains an asymptotic region symplectomorphic to $(\mathbb{C}^2\setminus \overline {B_{\mathfrak{R}}},\omega_0)/\Gamma$ for some sufficiently large common radius $\mathfrak{R}$. On the other hand, the base point of the $\Gamma$-capsule has a neighbourhood symplectomorphic to $(B_\mathfrak{r},\omega_0)/\Gamma$ for some small common radius $\mathfrak{r}$, and by shrinking the radius if necessary we can guarantee that this ball-quotient in the $\Gamma$-capsule does not meet some chosen symplectically immersed $2$-sphere satisfying \ref{3}. We now inflate the $\Gamma$-capsule by replacing the symplectic form $\omega_\Gamma$ by $t^2\omega_\Gamma$ for some large $t>0$. In the inflated $\Gamma$-capsule, the base point now has a neighbourhood symplectomorphic to $(B_R(0),\omega_0)/\Gamma$, where $R=t\mathfrak{r}$. Therefore by taking $t$ to be sufficiently large, we may arrange that $R>\mathfrak{R}$. By now removing $\overline{B_\mathfrak{R}(0)}/\Gamma\in p$ from the $\Gamma$-capsule and $(\mathbb{C}^2\setminus \overline{B_R(0)})/\Gamma$ from $M_\infty$, we are then left with pieces that can be glued together symplectically along copies of $(B_R(0)\setminus \overline{B_\mathfrak{R}(0)})/\Gamma$ to produce a compact symplectic manifold $(\widehat{M},\widehat{\omega})$.
\subsection{Rescue à la McDuff–Taubes}
Now that we have obtained a closed symplectic manifold $(\widehat{M},\widehat{\omega})$, we may invoke deep results from four–dimensional symplectic geometry due to McDuff \cite{McDuff} and Taubes \cite{TaubesBook}. The proofs of Theorems~\ref{16}--\ref{18} are identical to their K\"ahler counterparts (theorem~3.1, proposition~3.3, proposition 3.4 in \cite{LeBrun100}), and we therefore refer the reader to \cite{LeBrun100} for the detailed arguments. The crucial point is that LeBrun’s method relies only on the symplectic capping-off construction, rather than on the integrability of the almost-complex structure. Consequently, the argument extends to the almost K\"ahler setting, provided that $J$ satisfies the decay condition
$
|\nabla^{\mathrm{LC}} J| 
= O\big(r^{-\frac{\eta}{2}}\big)$ as $r \longrightarrow \infty,
$
for some $\eta > 4$, as established above.\par 
The proofs of theorems \ref{Penrose} and \ref{PMT} can also be deduced from their K\"ahler counterparts in \cite{LeBrun100}. For the convenience of the reader, we briefly outline the main ideas in the proof of theorem~\ref{Penrose} (Penrose inequality), from which theorem~\ref{PMT} (positive mass theorem) immediately follows.
\begin{proof}[Proof of theorem \ref{Penrose}]
    As explained above, starting with an AE almost K\"ahler manifold $(M,g,J)$ with appropriate decay rate of $\nabla^{\mathrm{LC}}J$, the symplectic capping-off procedure produces a compact symplectic manifold $(\widehat{M},\widehat{\omega})$ by removing a standard symplectic end $\big(\mathbb{R}^4\setminus B_{\mathfrak{R}}(0),\omega_0\big)$ and replacing it with $\mathbb{CP}^2$ minus a ball, equipped with some multiple of the Fubini-Study symplectic form. In this setting, a projective line in $\mathbb{CP}^2$ gives us a symplectic $2$-sphere of self-intersection $+1$ in $(\widehat{M},\widehat{\omega})$.\par 
    Invoking corollary~1.5 from \cite{McDuff}, we deduce that $(\widehat{M},\widehat{\omega})$ is symplectomorphic to a blow up of $\mathbb{CP}^2$, equipped with some K\"ahler form, in a way that sends the given $2$-sphere to a projective line $\mathbb{CP}^1$ that avoids all the blown up points. Removing this ``\textit{line at infinity}", we this see that $M$ must be diffeomorphic to $\mathbb{R}^4\# k\overline{\mathbb{CP}^2}$, where $k=b_2(M)$, and $H_2(M,\mathbb{Z})$ is moreover generated by the homology classes of $k$ disjoint symplectic $2$-spheres $E_1,\dots,E_k\subset M$ of self-intersection $-1$.\par 
    Therefore, under the natural identification $H^2_c(M)=H_2(M,\mathbb{R})$ arising from the Poincare duality, we have $\iota^{-1} (-c_1)=\sum_{j=1} [E_j]$, as can be checked by integrating both sides against each of the homology generators $[E_j]$. Thus, the mass formula \eqref{mass formula} in the AE case gives us
    \begin{align*}
        \mathfrak{m}(M,g)=\frac{1}{3\pi}\sum_{j=1}^k\int_{E_j}\omega+\frac{1}{12\pi^2}\int_M \Big(s+\frac{1}{4}|\nabla^{\mathrm{LC}}J|^2\Big) d\mathrm{v}_g
    \end{align*}
The main remaining step in the proof is to show that each of the homology classes $E_j$ can in fact be represented by a finite sum of $J$–holomorphic curves 
with positive integer multiplicities. This is precisely the technical heart of the argument, and at this point we appeal to the proof of theorem~3.5 (\S 3) in LeBrun \cite{LeBrun100}. The crucial observation is again that LeBrun’s reasoning does not depend on the integrability of the almost-complex structure, but only on the symplectic nature of the compactification and on Taubes’ ``SW=Gr" \cite{TaubesBook} correspondence in the case $b_2^+ = 1$ \cite{Li}. Since our compactified manifold $(\widehat{M},\widehat{\omega})$ is symplectic and satisfies $b_2^+ = 1$, Taubes' results apply verbatim in the almost K\"ahler setting.\par 
The classes $E_j$ extend naturally to homology classes in 
$H_2(\widehat{M},\mathbb{Z})$. By construction, these classes satisfy the same intersection-theoretic properties as in LeBrun’s setting: they are represented by symplectic exceptional classes with negative self-intersection and positive pairing with the symplectic form.\par
At this stage one invokes Taubes’ theorem identifying Seiberg–Witten basic classes with Gromov invariants on symplectic $4$–manifolds with $b_2^+ = 1$. The relevant Seiberg–Witten invariants are non-trivial, and therefore each such class $E_j$ is represented by a (possibly singular) $J$–holomorphic curve with positive integer multiplicities. One more key point in LeBrun's proof is that one needs to do the symplectic compactification a bit more carefully than we explained above because we would want these $J$-holomorphic curves coming from Taubes' powerful machinery to exist completely inside the image of $M$ inside $\widehat{M}$. The argument uses only symplectic topology and Taubes’ correspondence and not the integrability of $J$, hence it carries over unchanged to the present almost K\"ahler context.\par 
Let $\mathcal{E}_j$ be the $J$-holomorphic curve which represents the homology class $[E_j]$. This means that $\int_{E_j}\omega$ is in fact exactly the area of $\mathcal{E}_j$ and our mass formula therefore can be written as 
\begin{align*}
    \mathfrak{m}(M,g)=\frac{1}{3\pi}\sum_j\mathrm{vol} (\mathcal{E}_j)+\frac{1}{12\pi^2}\int_M \Big(s+\frac{1}{4}|\nabla^{\mathrm{LC}}J|^2\Big)\, d\mathrm{v}_g
\end{align*}
If the $D_i$ are the various spherical components of the various $E_j$, and if $n_i$ is the multiplicity with which a given $D_i$ occurs in this way, we can then rewrite this as
\begin{align*}
    \mathfrak{m}(M,g)=\frac{1}{3\pi}\sum_i n_i\,\mathrm{vol}(D_i)+\frac{1}{12\pi^2}\int_M \Big(s+\frac{1}{4}|\nabla^{\mathrm{LC}}J|^2\Big)\, d\mathrm{v}_g
\end{align*}
Since, $s\geq 0$, we get the Penrose-type inequality
\begin{align*}
    \mathfrak{m}(M,g)\geq \frac{1}{3\pi}\sum_i n_i \,\mathrm{vol}(D_i)
\end{align*}
The equality case happens when $s\equiv 0,\nabla^{\mathrm{LC}}J\equiv 0$, i.e., $(M,g,J)$ is scalar-flat K\"ahler.
\end{proof}
\begin{proof}[Proof of theorem \ref{PMT}]
    Theorem \ref{Penrose} implies that $\mathfrak{m}(M,g)\geq 0$. If the mass is zero, $\nabla^{\mathrm{LC}}J\equiv 0$ is forced. Hence we descend to the K\"ahler case and the theorem follows from its K\"ahler counterpart corollary~3.6 in \cite{LeBrun100}.
\end{proof}

\section{Rigidity results}\label{section 3}
In the compact case, theorem \ref{dim >4} is established in \cite{Sekigawa2}. The order of the ALE metric and the decay property of $|\nabla^{\mathrm{LC}}J|$ mentioned in the statement of theorem \ref{dim >4} is enough to compensate for the non-compactness of the manifold and the same argument carries over in the ALE case considered here. The main ingredient of the proof in \cite{Sekigawa2} is a remarkable \textit{integral formula} found by Sekigawa. We use a version of the integral formula written in \cite{Apostolov,Apostolov2} for our proof.\par
We begin with some definitions. The almost complex structure $J$ splits the space of real $2$-forms into $J$-invariant and $J$-anti-invariant ones. For any real $2$-form, we shall use the super-script $'$ to the sub-bundle of $J$-invariant $2$-forms, while the super-script $''$ stands for the projection to the sub-bundle of $J$-anti-invariant ones. Thus for any $2$-form $\beta\in\Omega^2$,
\begin{align*}
    \beta'(\cdot,\cdot):=\frac{1}{2}\big(\beta(\cdot,\cdot)+\beta(J\cdot,J\cdot)\big)\,\,\mathrm{and}\,\, \beta''(\cdot,\cdot):=\frac{1}{2}\big(\beta(\cdot,\cdot)-\beta(J\cdot,J\cdot)\big)
\end{align*}
We denote $W''$ the component of the curvature operator defined by
\begin{align*}
    W''_{XYZT}:=\frac{1}{8}\Big(&R_{XYZT}-R_{JX\,JY\,ZT}-R_{XY\,JZ\,JT}+R_{JX\,JY\,JZ\,JT}\\ 
    &-R_{X\,JY\,Z\,JT}-R_{JX\,YZ\,JT}-R_{X\,JY\,JZ\,T}+R_{JX\,Y\,JZ\,T}\Big)
\end{align*}
We also define the \textit{twisted} Ricci form to be $\rho^*:=R(\omega)$. The trace of $\rho^*, s^*=2R(\omega,\omega)$ is the $*$-scalar curvature introduced before. Finally we define a $2$-form $\phi$ defined by $\phi(X,Y):=\langle \nabla^{\mathrm{LC}}_{JX}\omega,\nabla^{\mathrm{LC}}_{Y}\omega\rangle$.
\begin{lemma}[Apostolov-Dr\u aghici-Moroianu \cite{Apostolov}, Apostolov-Dr\u aghici \cite{Apostolov2}]
    Let $(M,g,J)$ be an almost K\"ahler Einstein manifold. The following relation holds.
    \begin{align}\label{Apostolov}
     8\delta \langle\rho^*,\nabla^{\mathrm{LC}}\omega\rangle-\Delta |\nabla^{\mathrm{LC}}\omega|^2=8|W''|^2+4|(\rho^*)''|^2+|\big((\nabla^{\mathrm{LC}})^*\nabla^{\mathrm{LC}}\omega\big)'|^2+|\phi|^2+\frac{s}{n}|\nabla^{\mathrm{LC}}\omega|^2
    \end{align}
     $\delta$ denotes the co-differential with respect to $\nabla^{\mathrm{LC}}$.
\end{lemma}
One can find the formula stated in the lemma in this exact form in proposition 2 in \cite{Apostolov2} or in proposition 1 in \cite{Apostolov}.
\begin{proof}[Proof of theorem \ref{dim >4}] We have
\begin{align*}
    |\nabla^{\mathrm{LC}} \omega|^2=\frac{1}{2}|\nabla^{\mathrm{LC}} J|^2 = s^* - s=2R(\omega,\omega)-s
\end{align*}
Therefore as $r\longrightarrow\infty$, we have
\begin{align*}
    |W''|&=O\big(r^{-(\tau+2)}\big)\\
    |(\rho^*)''|&=O\big(r^{-(\tau+2)}\big)\\
    |\big((\nabla^{\mathrm{LC}})^*\nabla^{\mathrm{LC}}\omega\big)'|&=O\big(|R|)=O\big(r^{-(\tau+2)}\big)\\
    |\nabla^{\mathrm{LC}}\omega|^2&=O\big(|R|)=O\big(r^{-(\tau+2)}\big)\\
    |\phi|&=O\big(|\nabla^{\mathrm{LC}}\omega|^2)=O\big(r^{-(\tau+2)}\big)\\
    s&=O\big(r^{-(\tau+2)}\big)
\end{align*}
Since $\tau>\frac{2}{3}(n-4)$, we have $2\tau+4>n$. Hence the quantity in the right hand side of the equality in equation \eqref{Apostolov} is in $L^1$. Moreover as $r\longrightarrow\infty$, we also have 
\begin{align*}
    \langle\rho^*,\nabla^{\mathrm{LC}}\omega\rangle&=O\big(|R||\nabla^{\mathrm{LC}}\omega|\big)=O\big(r^{-(\frac{3\tau}{2}+3)}\big)\\
    d|\nabla^{\mathrm{LC}}\omega|^2&=O\big(|(\nabla^{\mathrm{LC}})^2\omega||\nabla^{\mathrm{LC}}\omega|\big)=O\big(r^{-(\eta+\frac{\tau}{2}+1)}\big)
\end{align*}
Since $\eta\geq\frac{2}{3}(n-1)$ and $\tau>\frac{2}{3}(n-4)$, we have $(\eta+\frac{\tau}{2}+1),(\frac{3\tau}{2}+3)>n-1$. Thereafter the following boundary integral vanish.
\begin{align*}
    \int_M\big(8\delta \langle\rho^*,\nabla^{\mathrm{LC}}\omega\rangle-\Delta |\nabla^{\mathrm{LC}}\omega|^2\big) d\mathrm{v}_g=\lim_{r\rightarrow \infty}\int_{S_r/\Gamma}\big(8 \langle\rho^*,\nabla^{\mathrm{LC}}\omega\rangle-d |\nabla^{\mathrm{LC}}\omega|^2\big) d\mathrm{v}_{S_r/\Gamma}=0
\end{align*}
Now integrating \eqref{Apostolov} over $M$ we obtain the theorem.
\end{proof}
Next we discuss the proof of theorem \ref{theorem 2}. The proof is essentially due to LeBrun \cite{LeBrun}. We compensate the non-compactness of the manifold with suitable decay properties of the curvature. Below we provide a brief outline of the proof, emphasizing where the decay assumptions enter. Before we delve into the details, we state a lemma gathering all the algebraic facts needed to prove the theorem. \par 
    \begin{lemma}[LeBrun \cite{LeBrun}]
        On an almost K\"ahler $4$-manifold $(M,g,J)$, we have the following identities.
        \begin{align}
            \frac{1}{2}|\nabla^{\mathrm{LC}}\omega|^2=W_+(\omega,\omega)-\frac{s}{3}\label{400}\\
            \big\langle W_+,(\nabla^{\mathrm{LC}})^*\nabla^{\mathrm{LC}}(\omega\otimes\omega)\big\rangle d\mathrm{v}_g=\big(W_+(\omega,\omega)\big)^2+4|W_+(\omega)|^2-sW_+(\omega,\omega)d\mathrm{v}_g\label{401}\\ 
            4|W_+|^2-4|W_+(\omega)|^2+\frac{1}{2}\big(W_+(\omega,\omega)\big)^2\geq 0\label{500}
        \end{align}
        Moreover if $\delta W_+=0$, we have 
        \begin{align}
            0=(\nabla^{\mathrm{LC}})^*\nabla^{\mathrm{LC}} W_++\frac{s}{2}W_+-6W_+\circ W_++2|W_+|^2\mathrm{Id}\label{402}
        \end{align}
    \end{lemma}
    \begin{proof}
        See §~2 of \cite{LeBrun}.
    \end{proof}
\begin{proof}[Proof of theorem \ref{theorem 2}]
    Notice as $r\longrightarrow\infty$,
    \begin{align*}
        |W_+|=O \big(r^{-(2+\tau)}\big)\\
        |\nabla^{\mathrm{LC}}W_+|=O \big(r^{-(3+\tau)}\big)\\
        s=O \big(r^{-(2+\tau)}\big)\\ 
        |\nabla^{\mathrm{LC}}\omega|^2=O \big(r^{-(2+\tau)}\big)
    \end{align*}
    We have used equation \eqref{400} for the asymptotics of $\nabla^{\mathrm{LC}}\omega$. Notice since $\tau>0$, all the terms $|W_+|^2, sW_+, |\nabla^{\mathrm{LC}}\omega|^4, s|\nabla^{\mathrm{LC}}\omega|^2,s^2\in L^1$.\par 
    Equation \eqref{402} gives us
    \begin{align}\label{eq 403}
        0=\int_M \big\langle (\nabla^{\mathrm{LC}})^*\nabla^{\mathrm{LC}} W_++\frac{s}{2}W_+-6W_+\circ W_++2|W_+|^2\mathrm{Id},\omega\otimes\omega\big\rangle d\mathrm{v}_g
    \end{align}
    We focus on the first term of the integrand. Integrating by parts we get 
    \begin{align*}
        \int_M \big\langle (\nabla^{\mathrm{LC}})^*\nabla^{\mathrm{LC}} W_+,\omega\otimes\omega\big\rangle d\mathrm{v}_g=\int_M \big\langle \nabla^{\mathrm{LC}} W_+,\nabla^{\mathrm{LC}}(\omega\otimes\omega)\big\rangle d\mathrm{v}_g+\text{ a boundary term}
    \end{align*}
    As $r\longrightarrow\infty$,
    \begin{align*}
    |\nabla^{\mathrm{LC}} W_+||\omega\otimes\omega|=O (r^{-(3+\tau)})
    \end{align*}
    Since $\tau>0,(3+\tau)>3$, the boundary term vanishes and we have 
    \begin{align*}
        \int_M \big\langle (\nabla^{\mathrm{LC}})^*\nabla^{\mathrm{LC}} W_+,\omega\otimes\omega\big\rangle d\mathrm{v}_g=\int_M \big\langle \nabla^{\mathrm{LC}} W_+,\nabla^{\mathrm{LC}}(\omega\otimes\omega)\big\rangle d\mathrm{v}_g
    \end{align*}
    Integrating by parts again we get
    \begin{align*}
        \int_M \big\langle \nabla^{\mathrm{LC}} W_+,\nabla^{\mathrm{LC}}(\omega\otimes\omega)\big\rangle d\mathrm{v}_g= \int_M \big\langle W_+,(\nabla^{\mathrm{LC}})^*\nabla^{\mathrm{LC}}(\omega\otimes\omega)\big\rangle d\mathrm{v}_g+\text{ a boundary term}
    \end{align*}
    Again we see that as $r\longrightarrow\infty$,
    \begin{align*}
        |W_+||\nabla^{\mathrm{LC}}(\omega\otimes\omega)|=O (|W_+||\nabla^{\mathrm{LC}}\omega|)=O(r^{-(3+\frac{3\tau}{2})})
    \end{align*}
    Since $\tau>0,3+\frac{3\tau}{2}>0$, the boundary term vanishes again. Therefore equation \eqref{eq 403} can be written as
    \begin{align}\label{eq 404}
        0=\int_M \bigg(\big\langle W_+,(\nabla^{\mathrm{LC}})^*\nabla^{\mathrm{LC}}(\omega\otimes\omega)\big\rangle +\frac{s}{2}W_+(\omega,\omega)-6|W_+(\omega)|^2+2|W_+|^2|\omega|^2\bigg) d\mathrm{v}_g
    \end{align}
    Substituting the identity \eqref{401} for $\big\langle W_+,(\nabla^{\mathrm{LC}})^*\nabla^{\mathrm{LC}}(\omega\otimes\omega)\big\rangle$ into equation \eqref{eq 404} yields
    \begin{align}\label{eq 501}
        0=\int_M \Big(\big(W_+(\omega,\omega)\big)^2-\frac{s}{2}W_+(\omega,\omega)-2|W_+(\omega)|^2+4|W_+|^2\Big)d\mathrm{v}_g
    \end{align}
    Combining the above equation \eqref{eq 501} with the inequality \eqref{500} we have
    \begin{align*}
        \int_M sW_+(\omega,\omega) d\mathrm{v}_g\geq \int_M \bigg(4|W_+|^2+\frac{3}{2}\big(W_+(\omega,\omega)\big)^2\bigg)d\mathrm{v}_g
    \end{align*}
    which also reads
    \begin{align*}
        \frac{3}{8}\int_M \bigg(\frac{2s}{3}-W_+(\omega,\omega)\bigg)W_+(\omega,\omega)d\mathrm{v}_g\geq \int_M |W_+|^2 d\mathrm{v}_g
    \end{align*}
    Substituting $W_+(\omega,\omega)=\frac{1}{2}|\nabla^{\mathrm{LC}}\omega|^2+\frac{s}{3}$ from \eqref{400} into the above equation we get
    \begin{align*}
        \frac{3}{8}\int_M \bigg(\frac{s}{3}-\frac{1}{2}|\nabla^{\mathrm{LC}}\omega|^2\bigg)\bigg(\frac{s}{3}+\frac{1}{2}|\nabla^{\mathrm{LC}}\omega|^2\bigg) d\mathrm{v}_g\geq \int_M |W_+|^2 d\mathrm{v}_g
    \end{align*}
    An algebraic simplification yields 
    \begin{align}\label{ineq 1}
        \int_M\frac{s^2}{24} d\mathrm{v}_g-\frac{3}{32}\int_M|\nabla^{\mathrm{LC}}\omega|^4 d\mathrm{v}_g\geq \int_M |W_+|^2 d\mathrm{v}_g
    \end{align}
    On the other hand using the fact that $W_+$ is trace-free and $|\omega|^2\equiv 2$, we have 
    \begin{align*}
        2\sqrt{\frac{2}{3}}|W_+|\geq W_+(\omega,\omega)
    \end{align*}
    Combining it with the identity \eqref{400} yields
    \begin{align*}
        |W_+|\geq \frac{s}{2\sqrt{6}}
    \end{align*}
    Since $s\geq 0$, squaring and then integrating both sides we get
    \begin{align}\label{ineq 2}
        \int_M |W_+|^2 d\mathrm{v}_g\geq \int _M \frac{s^2}{24}d\mathrm{v}_g
    \end{align}
    \eqref{ineq 1} and \eqref{ineq 2} together proves that $\nabla^{\mathrm{LC}}J\equiv 0$. 
\end{proof}
\begin{remark}
    To prove the analogous statements for ALF, one needs to make sure again that under certain decay assumptions, the boundary terms do not contribute whenever we do any integration by parts. An $n+1$-dimensional ALF manifold of order $\tau>0$ has the following asymptotics for the curvature tensor: $|R|=O(r^{-(2+\tau)})$ as $r\rightarrow\infty$. This is exactly like the ALE case except now $r$ represents the radius in $\mathbb R^n$ and not in $\mathbb R^{n+1}$ as in the ALE case; but that is okay, because the volume of a boundary hypersurface in an $n+1$-dimensional ALF space grows as $r^{n-1}$, one power less than the $n+1$-dimensional ALE case. Thereafter, all the same calculations go through the same. 
\end{remark}
\printbibliography[
heading=bibintoc,
title={Bibliography}
]

%
%
%


%

\Addresses
\end{document}